\documentclass[journal]{IEEEtran}
\usepackage{amsthm,amssymb,amsmath,rangecite}
\usepackage{amsmath,graphicx,epstopdf, bm, amsthm,dsfont,amssymb,float,cite,color,array,tikz}
\usepackage{booktabs}
\newtheorem{Theorem}{Theorem}

\newtheorem{proposition}[Theorem]{Proposition}

\newtheorem*{remark}{Remark}
\newtheorem{lemma}[Theorem]{Lemma}

\usepackage{lscape}
\usepackage{cite}
\usepackage{url}
\usepackage{hyperref}
\newcommand{\mfR}{{\mathfrak{R}}}

\newcommand{\xvec}{{\bf{x}}}

\newcommand{\Xmat}{{\bf{X}}}

\newcommand{\Prob}{\mathbb{P}}
\newcommand{\Ex}{\mathbb{E}}
\newcommand{\mJ}{\mathcal{J}}

\newcommand{\Wnm}{W_n^{(m)}}

\newcommand{\define}{\stackrel{\triangle}{=}}

\newcommand{\be}{\begin{equation}}
\newcommand{\ee}{\end{equation}}
\newcommand{\beqna}{\begin{eqnarray}}
\newcommand{\eeqna}{\end{eqnarray}}

\begin{document}
	%\bstctlcite{IEEEexample:BSTcontrol}
	\title{Bayesian Quickest Detection of Propagating Spatial Events}
	\author{Topi~Halme,~\IEEEmembership{Student~Member,~IEEE,}
		Eyal~Nitzan,~\IEEEmembership{Member,~IEEE,}
		and~Visa~Koivunen,~\IEEEmembership{Fellow,~IEEE}% <-this % stops a space
		\thanks{This work was partly supported by the Academy of Finland projects: (1) Statistical Signal Processing Theory and Computational Methods for Large Scale Data Analysis (2) WiFiUS project: Secure Inference in the Internet of Things.}
		\thanks{{\footnotesize{ T. Halme, E. Nitzan, and V. Koivunen are with the Department of Signal Processing and Acoustics, Aalto University, Espoo, Finland, e-mail: topi.halme@aalto.fi, eyal.nitzan@aalto.fi,  and visa.koivunen@aalto.fi.}}}
	}
	
	\maketitle
	\nopagebreak
	
	\begin{abstract}
		Rapid detection of spatial events that propagate across a sensor network is of wide interest in many modern applications. In particular, in communications, radar, IoT, environmental monitoring, and biosurveillance, we may observe propagating fields or particles. In this paper, we propose Bayesian sequential single and multiple change-point detection procedures for the rapid detection of such phenomena. 
%		It is assumed that the spatial event propagates across a network of sensors according to the physical properties of the source causing the event. 
Using a dynamic programming framework we derive the structure of the optimal single-event quickest detection procedure, which minimizes the average detection delay (ADD) subject to a false alarm probability upper bound. The multi-sensor system configuration is arbitrary and sensors may be mobile. In the rare event regime, the optimal procedure converges to a more practical threshold test on the posterior probability of the change point. A convenient recursive computation of this posterior probability is derived by using the propagation characteristics of the spatial event. The ADD of the posterior probability threshold test is analyzed in the asymptotic regime, and specific analysis is conducted in the setting of detecting random Gaussian signals affected by path loss. Then, we show how the proposed procedure is easy to extend for detecting multiple propagating spatial events in parallel in a multiple hypothesis testing setting. A method that provides strict false discovery rate (FDR) control is proposed. In the simulation section, it is demonstrated that exploiting the spatial properties of the event decreases the ADD compared to procedures that do not utilize this information, even under model mismatch.
	\end{abstract}
	
	
	\section{Introduction}
	Sequential change-point detection, often referred to as quickest detection, is a fundamental statistical inference task \cite{PAGE,SHIRYAEV_OPTIMUM,LORDEN,POLLAK,MOUSTAKIDES,LAI,POOR_QUICKEST,TARTAKOVSKY_GENERAL,TARTAKOVSKY_ASIMPTOTIC}. It is encountered in numerous applications, such as Internet of Things, environmental monitoring, biosurveillance, finance, radar, and wireless communications. Sensor networks are commonly used to rapidly detect a disruption or an event in the monitored physical enviornment \cite{RAGHAVAN,CHEN_ZHANG_POOR_BAYESIAN_JOURNAL,VEERAVALLI_DECENTRALIZED,POOR_ONE_SHOT,CHAUDHARI,KURT_WANG}. Usually, in these sensor networks the sensors communicate with a fusion center (FC) or a cloud that performs statistical inference tasks based on the data or local statistics from the sensors. The network setting can be centralized \cite{RAGHAVAN,CHEN_ZHANG_POOR_BAYESIAN_JOURNAL} where the FC has access to all the data from the sensors or decentralized \cite{VEERAVALLI_DECENTRALIZED,POOR_ONE_SHOT,CHAUDHARI,KURT_WANG,JIAN_LI_DISTRIBUTED,CUI_DISTRIB} where the sensors perform local computations/inferences and may only send the results, e.g. some sufficient statistic, to the FC. Recently, quickest detection of multiple change points in parallel has gained wide interest \cite{CHEN_ZHANG_POOR_NON_BAYESIAN,CHEN_ZHANG_POOR_BAYESIAN_JOURNAL,CISS2020,ICASSP2020,NITZAN_TSP_BAYESIAN}. Parallel multiple change points can be caused, for example, by multiple active radio transmitters, multiple sound sources, multiple radar targets, or multiple emitters of polluting particles. 
%	For multiple change-point detection, the false discovery rate (FDR) is an appropriate criterion for control of the rate of false alarms \cite{CHEN_ZHANG_POOR_BAYESIAN_JOURNAL,NITZAN_TSP_BAYESIAN}. A false alarm (false discovery) in the change-point detection context occurs if the procedure declares a change before it has happened. The FDR is the expected proportion of the number of false discoveries among all discoveries \cite{BENJAMINI_HOCHBERG,EFRON_BOOK}.\\
	\indent
	
	In many cases, the event causing the change in the environment has some spatial properties. The event can be a moving target that appears in a surveillance system, propagating radio frequency or audio signals impinging distributed sensors, pollution emanating from a malfunctioning device, or the onset of an epidemic. The effect of the spatial event on a specific area of the network can be modeled in different ways depending on the underlying physical phenomenon of interest and parameters of the sensor system observing it, such as the sensor system configuration and sampling rate. For example, the event may instantaneously affect the sampling distributions of all the sensors in the vicinity of the event \cite{VEERAVALLI_DECENTRALIZED,TARTAKOVSKY_GENERAL}. Alternatively, the disruption may propagate across the sensors in the network or in some cluster of the network over a short time period \cite{RAGHAVAN,KURT_WANG,LUDKOVSKI}. Examples are propagation of polluting particles in environmental monitoring, seismic activity in earthquake monitoring, propagation of radio waves through space in communication or radar systems, and epidemic traveling wave in biosurveillance.\\
	\indent
	Several works have considered quickest detection while incorporating spatial information. In \cite{RAGHAVAN}, Bayesian quickest detection was considered where the sensors are numbered and located with regular geometry and uniform displacements. The initial origin of the disruption was known to be at the first sensor and the disruption propagates through all the sensors as a Markov process in an order determined by the numbering. Extension of \cite{RAGHAVAN} to the case where the first sensor experiencing the initial change is unknown was proposed in \cite{LAI_QUICKEST_MARKOV}. Given the sensor observing the disruption first, a predetermined change propagation trajectory was assumed across the sensors. The work in \cite{KURT_WANG} considered a similar setting to \cite{RAGHAVAN,LAI_QUICKEST_MARKOV}, with the difference that the change propagation pattern is assumed to be unknown. In addition, both centralized and decentralized settings were considered. In \cite{LUDKOVSKI}, Bayesian continuous-time single change-point detection with sensor networks was studied. The event was assumed to occur at a random time instant in a random location and gradually propagate through the sensor network with unknown velocity triggering interdependent change points. A numerical procedure was proposed to approximate the optimal but intractable Bayesian solution based on the approximated posterior probabilities. Non-Bayesian change-point detection with spatial information under different setups has been studied in \cite{SIEGMUND_YAKIR,XIE_COMMUNITY,XIE_EVENTS,ROVATSOS_SEQUENTIAL,ROVATSOS_ARXIV,ROVATSOS_ICASSP2020,ZOU_VEERAVALLI,WANG_SPATIAL,XIAN_ONLINE,NEBHAN_CORRELATION}. In \cite{ROVATSOS_SEQUENTIAL}, a setting where a moving anomaly affects one sensor at a time was considered. As the disruption moves over the network, the affected sensor changes with time according to either a known or an unknown probability model.
	Rapid detection of propagating phenomena has recently also been studied within the learning and adaptation framework \cite{MARANO}. In \cite{MARANO}, a fully-flat network without a central unit is deployed for monitoring. Sensors must exchange information with their neighbors in order to accurately estimate the true state of nature in their vicinity. Due to the more complex network communication topology and interaction and information exchange among the neighboring sensor nodes, obtaining strong performance guarantees in terms of detection delay and false alarm rate, as is the goal under a quickest detection formulation, is difficult. Moreover, the delays in exchanging information among sensors may be significant compared to the propagation speed of the monitored phenomenon. Nonetheless, an information diffusion scheme that results in each sensor detecting changes faster than sensors working in isolation was derived.
	\\
	\indent

	%In \cite{LUDKOVSKI}, Bayesian continuous-time single change-point detection with sensor networks was considered. The event was assumed to occur at a random time instant in a random location and gradually propagate through the sensor network with random velocity triggering interdependent change points. 
	%%Prior distributions were assumed for the initial event time and location and for the propagation velocity of the event. 
	%A numerical procedure was proposed to approximate the optimal but intractable Bayesian solution based on the approximated posterior probabilities. In this paper the setting is discrete-time
	
	% One of the propagation dynamics considered is a wavefront with increasing radius as can be seen in Fig. \ref{wavefront}. This propagation model is relevant, for example, to biosurveillance applications that attempt to detect outbreaks of epidemics \cite{LUDKOVSKI,LOPEZ_MODELING,TROVAO_INFERENCE,CUNNIFFE_OAK}. Another application is in wireless communications where radio waves that carry information are emitted from a transmitter, propagate through space with velocity equal to speed of light, and detected by a receiver \cite[Ch. 3]{RADIO_BOOK}. The considered propagation model and change detection setup are also relevant in radar and seismic monitoring and localization applications where, for example, a point source or a target is generating or reflecting waveforms that propagate across a sensor array as a plane wave \cite{YONG_POINT,WEISS_POINT}.\\
	\indent
	In this paper, we propose a Bayesian method for the detection of a propagating spatial event. A discrete-time model is used in acquiring observations. It is assumed that the event propagates in a two-dimensional plane  with area radius that increases either randomly or deterministically based on the laws of physics. This wavefront propagation model is illustrated in Fig. \ref{wavefront}. It is relevant, for example, in wireless communications where radio waves that carry information are emitted from a transmitter, propagate through space with velocity equal to the speed of light, and detected by a receiver \cite[Ch. 3]{RADIO_BOOK}. The considered propagation model and change detection setup are also relevant in radar and seismic monitoring and localization applications where, for example, a point source or a target is generating or reflecting waveforms that propagate across a sensor array as a plane wave \cite{YONG_POINT,WEISS_POINT}. Another application is in biosurveillance applications that attempt to detect outbreaks of epidemics \cite{LUDKOVSKI,LOPEZ_MODELING,TROVAO_INFERENCE,CUNNIFFE_OAK}. 
	%The highly relevant in practice and commonly encountered special case of propagation with constant radial velocity is considered as well. 
	The sensors of the network take observations sequentially in discrete time slots. At each time slot, sensors that are located outside the disruption area obtain observations that follow a common null distribution (no signal present). Sensors that are located within the disruption area obtain observations that obey alternative distributions that may be different among the exposed sensors. We are interested in detecting the initial event as quickly as possible subject to statistical constraints on the rate of false alarms. We assume that the number of sensors and their locations are known at each time slot but may change in time, e.g. a mobile wireless sensor network \cite{CORTES_MOBILE} where the sensors might correspond to e.g. smart phones or drones. Since the main focus of this work is on detection of spatially localized events, it is assumed that all sensors are able to communicate individually with a common FC or cloud. The FC could be a base station (BS) serving different users and sensors in its coverage area. Modern wireless systems such as 5G have reasonably small coverage areas because of the higher frequencies they are using.  Separate mobile access points that traverse the network and collect information from the sensors, as in the SENMA framework \cite{LANG_SENMA, GKHAN_SENMA}, are not required. Obviously, there exists a wide area of applications where such access points are useful, but providing guarantees on the detection performance may be difficult in those settings. 
	
	\begin{figure}[ht!]
		\centering
		\includegraphics[scale = 1]{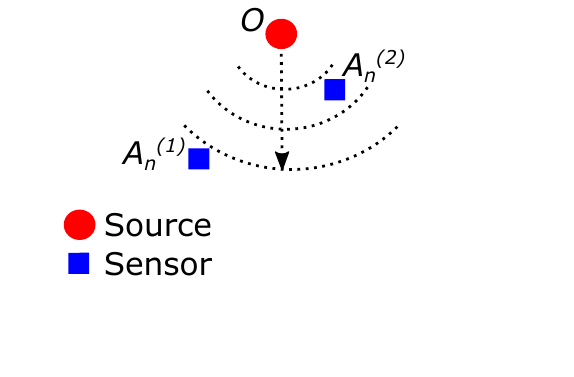}
		\caption{Source wavefront propagation: Phenomenon emanating from the source $O$ towards the sensors $A_n^{(1)}$ and $A_n^{(2)}$ where $n$ is the current time slot}
		\label{wavefront}
	\end{figure}
	
	The related works described earlier do not explicitly and jointly take into account the sensor locations, the displacement between the sensor location and the location of the disruption source, and potential sensor mobility. In particular, to the best of our knowledge, the problem of quickest detection in discrete time in the practically relevant scenarios where the change is caused by a gradually expanding spatially localized event(s) has not been addressed earlier. The most related previous work (\cite{RAGHAVAN, KURT_WANG, LAI_QUICKEST_MARKOV, ROVATSOS_SEQUENTIAL,ROVATSOS_ARXIV,ROVATSOS_ICASSP2020}) focuses on settings where the dynamics of the change-event are modeled as movement from \emph{sensor(s)-to-sensor(s)}. The particular assumed sensor network topology (e.g. an array \cite{RAGHAVAN} or a graph \cite{ROVATSOS_SEQUENTIAL}) results jointly from the spatial properties of the event, and the placement of the sensors. The event is then assumed to affect the sensors according to a model specified by this topology.
%The event itself may or may not have a spatial interpretation. 
In contrast, in this paper we consider potentially mobile sensor systems with completely arbitrary displacements and no regular sensing geometry. Therefore it is not in general possible to describe the propagation of the event with any fixed network topology or model. Whether a particular sensor is affected depends on its location relative to the source of the physical event, and the sensor locations may be arbitrary and change over time. Hence, the previous works cannot directly be applied to quickest detection of spatially phenomena emanating from a distinct and potentially unknown location, especially when sensor locations can vary in time.

%Despite the fact that the disruption is considered to have a clear physical world structure (Fig. \ref{wavefront}), it is not possible to model the propagation 

	The contributions of this paper are:
	\begin{itemize}
		\item We propose a dynamic programming framework for deriving a stopping time that exploits the spatial information and minimizes the average detection delay (ADD) under a constraint on the false alarm probability. An optimal detection procedure that minimizes the associated Bayes risk is derived. 
		\item As the optimal procedure is hard to implement in practice, we propose a more practical detection procedure based on thresholding the posterior probability of the change point having occurred. To theoretically justify the use of this simpler procedure, it is shown that the procedure is a limiting form of the optimal procedure in the rare event regime. A convenient recursive formula for computing the posterior probability is developed by using the radial change propagation pattern of the spatial event.
		\item The probability of false alarm (PFA) control of the thresholding procedure is established and its ADD is analyzed in the asymptotic regime. Conditions under which the threshold procedure achieves asymptotic optimality are provided. Furthermore, we analyze the procedure in the specific setting of quickest detection of attenuating random signals. An approximation for asymptotic ADD is provided, and asymptotic optimality is established.
		\item We extend the single-event detection procedure to the detection of multiple statistically independent spatial events in parallel by combining the developed posterior probability threshold procedure with a Multiple Hypothesis Testing (MHT) setup. The rate of false alarms is controlled using the False Discovery Rate (FDR) criterion, which is widely used for MHT \cite{CHEN_ZHANG_POOR_BAYESIAN_JOURNAL,BENJAMINI_HOCHBERG,EFRON_BOOK}. This is highly relevant in many modern applications where high dimensional data must be processed in parallel and there may be multiple events taking place at the same time \cite{CHEN_ZHANG_POOR_NON_BAYESIAN,CHEN_ZHANG_POOR_BAYESIAN_JOURNAL,NITZAN_TSP_BAYESIAN}. It is shown that the proposed parallel procedure strictly controls the FDR level.
		\item Simulations are conducted to verify the theoretical findings. It is clearly demonstrated that exploiting the spatial properties of the event decreases the ADD compared to procedures that do not utilize this information in both single and multiple cluster settings. This benefit is achieved even under model mismatch. It is demonstrated that the gain in performance is the largest when the event propagates sufficiently slowly compared to the sampling rate, or when the sensor displacements are large, and/or when the pre- and post-change probability models are different enough.
	\end{itemize}
	
	Preliminary results of this paper appear in conference papers \cite{HALME_ICASSP2021} and \cite{ASILOMAR2021}. This paper is organized as follows. In Section \ref{sec:Model and problem formulation}, we formulate the quickest detection problem. A dynamic programming framework for the detection is formulated in Section \ref{sec:Dynamic programming for optimal stopping time}. Under the radial change propagation setup, a change-point detection procedure and its extension to multiple parallel change-point detection are presented in Sections \ref{sec:Single change-point detection procedure for radial propagation} and \ref{sec:Extension to multiple change-point detection}, respectively. Our simulations and conclusions appear in Sections \ref{sec:Numerical simulations} and \ref{sec:Conclusion}, respectively.

	\subsection*{Notation}
	Scalar random variables are denoted by normal font capital letters, with the exception of the change point $t$, which is also a random variable. Scalar constants, such as realizations of random variables, are denoted by normal font lowercase letters, with the exception of $L$, $M$, $N$, $K$ and $\mfR$, which are constant integers, and $U_{n,m,r}$ which denotes an event. Boldface uppercase and lowercase letters are used for vector random variables and constants, respectively. For an integer $K$, we use $[K]$ to denote the set $\{0,1,...,K-1\}$ of cardinality $K$. 
	
	\section{Model and problem formulation}\label{sec:Model and problem formulation}
	We begin by describing the model for a single spatial change-point detection problem. The model, relevant terms, and notations for multiple change-point detection in parallel will be described in Section \ref{sec:Extension to multiple change-point detection}. Let $(\Omega,\mathcal{F},\Prob)$ denote a probability space, where $\Omega$ is the sample space, $\mathcal{F}$ is the $\sigma$-algebra generated by $\Omega$, and $\Prob$ is a probability measure. The expectation operator with respect to $\Prob$ is denoted by $\Ex$.\\
	\indent
%	We consider a potentially mobile network of sensors that acquire observations from the environment at discrete time instances and communicate with an FC or cloud. The sampling rate in time depends on the application and underlying physical phenomenon that is being monitored, for example radio wave propagation at certain frequency band.
	 At each time slot we have sensors in known locations but with arbitrary configuration within a domain of interest, $\mathcal{S}\subset\mathbb{R}^2$. The set of sensor locations at time slot $n$ is denoted by $\mathcal{A}_n$ and the corresponding number of sensors is $|\mathcal{A}_n|$. If $|\mathcal{A}_n|=0$, then there are no observations received by the FC at time slot $n$. Unless otherwise stated, in this paper the sensor locations are considered known and deterministic. 
	%Given a sequence of locations $\mathcal{A}$, the conditional probability measure $\Prob(\cdot | \mathcal{A})$ is denoted as $\Prob_\mathcal{A}$. \\
	
	We consider a centralized setting where every sensor communicates its observations or local decision statistics to the FC. At time slot $n$, the data transmitted by the sensors and received by the FC are random variables $X_n^{(a)},~a\in\mathcal{A}_n$. The realization of $X_n^{(a)}$ is denoted by $x_n^{(a)}$. In mobile scenarios, the location information $a$ is communicated to the FC in addition to the observation value. Alternatively, the FC can have a capability to reliably estimate the locations of the sensors. Uncertainty in the location estimate could be represented as a probability distribution, which could be averaged over in a Bayesian framework. However, for the purposes of this paper we assume for simplicity that reliable point-estimates of the sensor locations exist. We define the $|\mathcal{A}_n|\times 1$ data vector $\Xmat_n$ that contains all the observations transmitted at time slot $n$, including the locations at which the observations were obtained by mobile sensors. It is assumed that at time slot $n$ the FC has access to the current and past observations, $I_n\define(\Xmat_1,\ldots,\Xmat_n)$, where $I_0$ is the empty set.\\
	\indent
	At a random time instant, $t$, a source becomes active and starts emitting a propagating signal/event from an unknown origin, $O$, causing a disruption in the domain of interest. It is assumed that the initial event time, $t$, has a geometric prior distribution with parameter, $\rho\in(0,1)$, i.e.
	\be\label{geom_prior}
	\Prob(t=m)=\rho(1-\rho)^{m},~m\in\mathbb{N}_0,
	\ee
	where $\mathbb{N}_0\define\mathbb{N}\bigcup\{0\}$ and $\mathbb{N}$ is the set of positive integers. The geometric prior distribution is very common in change-point detection because it is a mathematically convenient memoryless distribution, which is also relevant in a variety of practical applications  \cite{POOR_QUICKEST,TARTAKOVSKY_GENERAL,RAGHAVAN,KURT_WANG}.\\
	\indent
	We want to discover the initial event time, $t$, with minimal delay while controlling the PFA. In a multiple change-point setup, as will be described in Section \ref{sec:Extension to multiple change-point detection}, the PFA is replaced by the FDR criterion in a MHT framework. Hereafter, we refer to $t$ as the change point even though it may not cause an instantaneous change in the received observations as in classic change-point detection, due to the fact that the sensors are in distinct locations and displaced from the signal source. 
%	The behavior of the spatial event is represented by a propagating wavefront model \cite{LUDKOVSKI}, as shown in Fig. \ref{wavefront}, where the wave propagation through space has some radial velocity, for example, the speed of light for electromagnetic waves.\\
	\indent
	
	In order to reduce the complexity of the problem, we initially assume that $O \in \mathcal{O}\subset\mathcal{S}$ where $\mathcal{O}\define \{o_0,\ldots,o_{M-1}\}$ is a finite set of possible source or emitter locations in $\mathcal{S}$ with cardinality $|\mathcal{O}|=M$. We assume that the initial event can occur in any of the possible locations, $o_m\in\mathcal{O}$, with equal probability
	\begin{equation*}
	\Prob(O = o_m) = \frac{1}{M}.
	\end{equation*}
The assumption of a uniform probability distribution is made for simplicity and is not necessary for the following derivations. In practical settings it may not be realistic to expect that the source location can only appear within a known finite set of points. However, choosing the set $\mathcal{O}$ to be a sufficiently dense discretization of the 2D-plane can allow one to approximate the continuous field well and reduce the complexity, as is demonstrated in the simulation section. In this work the true location is initially assumed to lie within a known finite set in order to facilitate a dynamic programming solution. Additional signal processing may be applied to obtain point estimates of the location. That thoroughly studied source localization topic is outside the scope of this paper. \\
	
	The chosen sampling rate and duration of the discrete time slot used in acquiring the observations can highly affect the sensitivity of the network to the spatial event, and the time resolution and delay of detecting the change. Generally, the time-domain sampling rate should be selected so that one can distinguish among differences in the disruption arrival times at different sensors. In the considered model, if the event propagates with a constant radial velocity, we model the sampling rate such that during each time slot the radius of the disruption area increases by a fixed unit, e.g. some fraction or multiple of the wavelength. In order to take a variety of random propagation effects into account, we allow some randomness in the propagation of the spatial event. For example, epidemic spread may have a high degree of stochasticity due to random movements and interactions among individuals \cite{LUDKOVSKI}. Generally, propagation randomness can be due to randomness in the velocity \cite[Eq. (6.8)]{LUDKOVSKI}, due to timing jitter \cite{LOVELACE_JITTER}, or due to reflections, non-homogenous medium, scattering, and multipath \cite{RADIO_BOOK}.\\
	\indent
	Let $R_n$ denote the area radius of the propagating event at time slot $n$. For simplicity, we assume that the area radius can have only discrete integer values corresponding to a fixed distance unit. Let $\mathfrak{R} \in \mathbb{N}$ denote the smallest disruption area radius that covers the entire domain of interest, $\mathcal{S}$, regardless of the actual point of origin, $O\in\mathcal{O}$. Thus, we assume that $R_n\in [\mathfrak{R}+1]$. It is assumed that $R_n = 0$ when $n < t$, and $R_t = 1$,
	%\be\label{initial_event}
	%\Prob(R_n = 0 | t>n) = \Prob(R_{t} = 1) = 1.
	%\ee
	i.e. only when the initial change occurs, the event area radius expands by one unit. In addition, we assume that
	\be\label{R_n_1_to_R_n}
	\Prob(R_{n} = r+1 | R_{n-1} = r) = 1 - \Prob(R_{n} = r | R_{n-1} = r) = \rho_1,
	\ee
	$\forall r\in[\mfR]\setminus\{0\},~n\in\mathbb{N}$. At each time slot after the initial change the radius of the disruption increases by one radius unit with probability $\rho_1\in(0,1]$ and stays the same as in the previous time slot with probability $1-\rho_1$. As $\mfR$ is the maximum radius of the affected region, if $R_n = \mfR$ then $R_m = \mfR$ for $m \geq n$.
	%\be\label{R_n_1_to_R_n_max}
	%\Prob(R_{n} =\mfR | R_{n-1} = \mfR) = 1,~n\in\mathbb{N}.
	%\ee
	Using \eqref{geom_prior}, we obtain
	\be\label{R_0_1_to_R_0}
	\begin{split}
		\Prob(R_{n} = 1 | R_{n-1} = 0) &= 1 - \Prob(R_{n} = 0 | R_{n-1} = 0)\\ 
		&= \Prob(t=n|t\geq n) = \rho,~n\in\mathbb{N}.
	\end{split}
	\ee
	In addition, we obtain $\Prob(R_{0} = 0)=1-\rho$ and $\Prob(R_{0} = 1)=\rho$.\\
	\indent
	At each sensor location it is assumed that the sensor observes the disruption only if the disruption is present in this location, i.e. the distance between the sensor location and the source location is smaller than the current area radius of the disruption. Assume that the disruption is emanating from a source at $O=o_m$. Then, if a sensor at location $a\in\mathcal{S}$ is not exposed to the disruption, it acquires a noise-only observation coming from a known null probability density function (pdf), $f_0$. Otherwise, if this sensor is exposed to the disruption, it receives an observation with known pdf, $f_1^{(a,o_m)}$, that may depend on $a\in\mathcal{S}$ and $o_m\in\mathcal{O}$. For example, the power of the received signal can affect the parameters of the alternative pdf, and due to path loss may depend on the displacement between the sensor and the source \cite{RADIO_BOOK,CHITTE_RSS}. In many applications, the $f_0$ density represents random noise only, the statistical properties of which can be either known from theory, or estimated from training data even locally for each sensor in the absence of signal. On the other hand, the exact $f_1$ distribution, influenced by the appearing signal, may not always be known in practice. The issue of dealing with uncertainty in the $f_0$ and $f_1$ distributions has been an active topic of research in the field of quickest detection, see for example \cite{UNNIKRISHNAN, TARTAKOVSKY_UNKNOWN} and references therein. Therefore, in this work we consider the probability models to be known, and refer to the existing literature for solutions on handling any model uncertainty. 
%In our earlier work \cite{GOLZ} we proposed a bootstrap method for estimating the probability models in the pre-change state.
 Moreover, it will be observed in Section \ref{sec:Single change-point detection procedure for radial propagation} that for the PFA control it suffices to know only the $f_0$ distribution. In many detection problems, controlling the false positives is crucial so that the system is not overwhelmed with detections and subsequent tasks. Conditional on the true system state and the sensor locations, the observations at each time slot are assumed to be independent across the sensors, as well as independent of all previous observations. Since the individual sensors are distributed and in distinct locations, the sensor noise present in any physical measurement can be considered independent. \\
	\indent
	At each time slot, the FC decides whether the initial event has taken place or not based on the information, $I_n$, which is available at time slot $n$. To this end, it uses a stopping time, $T$, according to a predefined stopping rule. The delay in detection is quantified by the ADD,
	\be\label{ADD_margin}
	{\text{ADD}}(T)\define{\mathbb{E}}[(T-t)^+],
	\ee
	where $x^+ \define\max\{0,x\}$. The PFA is defined as
	\be\label{PFA}
	{\text{PFA}}(T)\define \Prob(T<t).
	\ee
	\section{Dynamic programming for optimal stopping time}\label{sec:Dynamic programming for optimal stopping time}
	In a similar manner to classic Bayesian change-point detection \cite{TARTAKOVSKY_GENERAL}, our goal is to derive the stopping time
	\be\label{radial_optimization}
	T_{\text{opt}}=\arg\underset{T\in\Delta_\alpha}{\inf}{\text{ADD}}(T),
	\ee
	where $\Delta_\alpha\define\{T:{\text{PFA}}(T)\leq\alpha\}$. Put into words, we want to find a stopping time with the smallest ADD among stopping times for which the PFA is not larger than $\alpha$, where $\alpha \in (0,1)$ is a predefined tolerated level of false alarms. In this section, we take a dynamic programming approach for solving \eqref{radial_optimization}. 
	
	\subsection{Finite horizon}
	We begin by restricting the stopping time to a finite horizon $[0, N]$. Solving the constrained optimization problem in \eqref{radial_optimization} can be approached by formulating a Lagrangian relaxation problem that minimizes the Bayes risk
	\begin{equation}\label{bayes_risk}
	B(T, c) \define  \Prob(T < t) + c\cdot\Ex[(T - t)^+]
	\end{equation}
	over all admissible stopping times. %In \eqref{bayes_risk}, $c$ is the cost of a unit of delay that balances the relative importance of detection accuracy and speed. 
	The state of the system at time $n$ is denoted by $S_n \in \{(m, r) : ~m \in [M],~r \in [\mfR+1]\}\cup\Upsilon$, with $S_n = (m,r)$ meaning that at time $n$ the event originating from $O=o_m$ has radius $R_n = r$. The term $\Upsilon$ represents the terminal state that the system goes into after a change is declared. In case $R_n = 0$, the spatial event has not occurred yet. From the description in Sec. \ref{sec:Model and problem formulation} it is clear that the system state $S_n$ evolves as a Markov process. Moreover, conditional on the system state and sensor locations, the observations are i.i.d. As such, the problem lends itself to a dynamic programming solution.\\
	\indent
	\indent
	Since $\{T < t\} \Leftrightarrow \{R_T = 0\}$, the Bayes risk in \eqref{bayes_risk} can be expressed in additive form as \cite{RAGHAVAN,KURT_WANG}
	\begin{equation}\label{finite_horizon_eq}
	B(T, c) =  \Prob\left(R_T = 0\right) + c \cdot \Ex\left[\sum_{n = 0}^{T-1}\Prob\left(R_n \geq 1\right)\right].
	\end{equation}
%	At time $n$, the state is not observable directly, but only through the observations up to this time, $I_n$, as described in Section \ref{sec:Model and problem formulation}. 
	In a finite horizon, we denote the minimum expected cost-to-go from $n$ to $N$ by $\mJ_n^N(I_n)$, which is in general a function of all available information $I_n$ at time $n$. The cost-to-go function obeys the backwards recursion
	\be\label{general_bellman}
	\begin{split}
		&\mJ_n^N(I_n) = \\
		&\min\big\{\Prob(R_n = 0 | I_n), c\cdot \Prob(R_n \geq 1 | I_n)+ \Ex [\mathcal{J}_{n+1}^N(I_{n+1}) |I_n]\big\},
	\end{split}
	\ee
	with 
	\begin{equation}\label{final_bellman}
	\mJ_N^N(I_N) = \Prob(R_N = 0 | I_N).
	\end{equation}
	In \eqref{general_bellman} the first term inside the minimum corresponds to the expected cost of stopping at $n$, and the second term denotes the expected cost of continuing the monitoring process. We denote the posterior probability of the event $\{S_n = (m,r)\}$ given $I_n$ by $p_{n,m,r}$. That is,
	\begin{equation}
	p_{n,m,r}  \define \Prob(O = o_m, R_n = r| I_n),
	\end{equation}
	and
	\begin{equation}
	\bm{p}_n \define [p_{n,0,0},...,p_{n,1,\mfR},p_{n,2,0},...,p_{n,2,\mfR},...,p_{n,M-1,\mfR}] 
	\end{equation} 
	is a $M \cdot (\mathfrak{R} + 1)$ dimensional vector that collects all of the probabilities of time $n$.
	In the next subsection, we present a recursive update formula for $\bm{p}_n$ that will be used in the dynamic programming solution.
	
	\subsection{Posterior probabilities computation}\label{subsec:posterior}
	At any time slot, $n$, the sample space of the considered setup, $\Omega$, can be partitioned as
	\be\label{Omega_partition}
	\Omega = \bigcup_{m=1}^M \bigcup_{r=0}^{\mfR} U_{n,m,r},
	\ee
	where 
	\be\label{event_partition}
	U_{n,m,r} \define \{O = o_m, R_n = r\} = \{S_n = (m,r)\}
	\ee
	are pairwise disjoint events and the events $\{O = o_m\}$ and $\{ R_n = r\}$ are independent. In the following, we derive a convenient recursive formula for computing $p_{n,m,r} = \Prob(U_{n,m,r} | I_n)$. Repeated use of the Bayes rule allows us to write $p_{n,m,r}$ as
	\be\label{Bayes}
	p_{n,m,r} = \frac{f(\xvec_n | U_{n,m,r})\Prob(U_{n,m,r}|I_{n-1})}{\sum_{l=1}^M \sum_{\tilde{r}=0}^{\mfR} f(\xvec_n | U_{n,l,{\tilde{r}}})\Prob(U_{n,l,{\tilde{r}}} | I_{n-1})}.
	\ee
	Given the conditional indepedence of the observations we have the factorization
	\be\label{observations_for_recursion}
	\begin{split}
		&f(\xvec_n | U_{n,m,r}) \\
		&= \prod_{a\in\mathcal{A}_n: \lVert a -  o_m\rVert < r}f_1^{(a,o_m)}(x_n^{(a)}) \prod_{a\in\mathcal{A}_n: \lVert a -  o_m\rVert \geq r} f_0(x_n^{(a)}).
	\end{split}
	\ee
	In addition, according to the assumed propagation model, $R_{n-1}$ can only be equal to $R_{n}$ or less than $R_{n}$ by one. Therefore, by the law of total probability, Bayes rule, and \eqref{R_n_1_to_R_n}-\eqref{R_0_1_to_R_0}, we can write
	\be
	\begin{split}
		\Prob(U_{n,m,r}|I_{n-1})&=\Prob(U_{n,m,r}|U_{n-1,m,r-1})p_{n-1,m,r-1}\\
		&~~~+\Prob(U_{n,m,r}|U_{n-1,m,r})p_{n-1,m,r}.
	\end{split}
	\ee
	The conditional probabilities that the radius increases by one radius unit during one time slot for different radius values are 
	\begin{equation*}
	\Prob(U_{n,m,r}|U_{n-1,m,r-1})=\rho_1, ~~ r\in [\mfR+1]\setminus\{0,1\}, 
	\end{equation*}
	and $\Prob(U_{n,m,1}|U_{n-1,m,0})=\rho$. The conditional probabilities that the radius stays the same during one time slot for different radius values are $\Prob(U_{n,m,\mfR}|U_{n-1,m,\mfR})=1$,
	\begin{equation*}
	\Prob(U_{n,m,r}|U_{n-1,m,r})=1-\rho_1,~\forall r\in [\mfR+1]\setminus\{0,\mfR\}, 
	\end{equation*}
	and $\Prob(U_{n,m,0}|U_{n-1,m,0})=1-\rho$. At $n=0$, we obtain 
	\begin{equation*}
	\Prob(U_{0,m,r})=0,~\forall r\in [\mfR+1]\setminus\{0,1\}, 
	\end{equation*}
	$\Prob(U_{0,m,1})=\frac{1}{M}\rho$, and $\Prob(U_{0,m,0})=\frac{1}{M}(1-\rho)$. In particular, it is seen that $\Prob(U_{n,m,r}|U_{n-1,m,r-j}),~j=0,1$, is independent of $n$. It should be noted that the radius of the area of the spatial event can reach a radius $r$ no earlier than time slot $n=r-1$.\\
	\indent
	At time $n$, the probabilities $\bm{p}_{n}$ can be updated using only the probabilities at the previous time step $\bm{p}_{n-1}$, current observation vector $\xvec_n$, and prior information. Thus, even as data accumulates with time, the amount of computations required for computing $\bm{p}_{n}$ remains constant per time slot. In particular, at time slot $n$ the amount of computations required for computing $\bm{p}_{n,m,r}$ is $\mathcal{O}(M\mfR)$.
	\indent
	
	It is observed that $\bm{p}_n$ depends on $I_{n-1}$ only through $\bm{p}_{n-1}$ and by \eqref{final_bellman} we have $\mathcal{J}_N^N(I_N) = \mathcal{J}_N^N(\bm{p}_N)$. Then, a simple induction argument shows that $\bm{p}_n$ is a sufficient statistic for the program, i.e. the minimum expected cost-to-go from $n$ to $N$ can be expressed as a function of $\bm{p}_n$, and thus $\mJ_n^N(I_n) = \mathcal{J}_n^N(\bm{p}_n)$. We denote the posterior probability of the event having radius $r$ at time $n$ by
	\be\label{posterior_cp}
	\pi_{n,r} \define  \Prob(R_n = r | I_n) = \sum_{m \in [M]} p_{n,m,r}.
	\ee
	The Bellman equations from \eqref{general_bellman} and \eqref{final_bellman} can then be expressed as 
	\begin{equation}\label{J_n_N_form}
	\mathcal{J}_n^N(\bm{p}_n) = \min(\pi_{n,0},~ c(1-\pi_{n,0}) + \mathcal{D}_n^N(\bm{p}_n)),
	\end{equation}
	and 
	\begin{equation}
	\mathcal{J}_N^N(\bm{p}_n) = \pi_{N,0},
	\end{equation}
	where 
	\begin{equation}
	\mathcal{D}_n^N(\bm{p}_n) \define \Ex[\mathcal{J}_{n+1}^N(\bm{p}_{n+1}) | I_n],
	\end{equation}
	can be expressed as a function of $\bm{p}_n$ similarly to \cite{RAGHAVAN,KURT_WANG}.
	
	\subsection{Extension to infinite horizon}
	In this subsection, we remove the upper bound on $T$, and consider the case $N \to \infty$. We write $\mathcal{J}_n \define \lim_{N \to \infty} \mathcal{J}_n^N$ for the cost-to-go function in the limit and similarly $\mathcal{D}_n \define \lim_{N \to \infty} \mathcal{D}_n^N$. The limits are well defined, since $0 \leq \mathcal{J}_n^N(\bm{p}_n) \leq 1$ and $\mathcal{J}_n^N(\bm{p}_n) \geq \mathcal{J}_n^{N+1}(\bm{p}_n)$ for all $\bm{p}_n, n$ and $N$. 
	Therefore, we obtain
	\begin{equation}
	\mathcal{J}_n(\bm{p}_n) = \min(\pi_{n,0}, c(1-\pi_{n,0}) + \mathcal{D}_n(\bm{p}_n)) \quad n\in\mathbb{N}_0,
	\end{equation}
	where $\mathcal{D}_n$ and $\mathcal{J}_n$ are non-negative functions on the $M \times (\mfR+1)$-dimensional simplex.
	It is then seen that the optimal stopping time $T_\text{opt}$ is of the form
	\begin{equation}\label{optimal_stop}
	T_\text{opt} = {\inf}\left\{n\in\mathbb{N}_0 : \pi_{n,0} < c(1-\pi_{n,0}) + \mathcal{D}_n(\bm{p}_n)\right\},
	\end{equation}
	where the change is declared the first time the posterior probability of the event not being present drops below $c(1-\pi_{n,0}) + \mathcal{D}_n(\bm{p}_n)$. In general, the structure of $\mathcal{D}_n$ is not explicitly known, hence no closed-form optimal solution exists. Furthermore a numerical approximation of the optimal stopping time is computationally challenging and may be hard to analyze.\\
	\indent
	An interesting special case is the regime where the initial disruption is a rare event, i.e. $\rho \to 0$. The following result establishes that in this scenario the optimal test $T_\text{opt}$ converges in probability to a simple threshold test on $\pi_{n,0}$ which provides an attractive solution for practical use.
	\begin{Theorem}\label{TRP_CONV}
		The optimal stopping rule in \eqref{optimal_stop} converges in probability to a threshold test
		\begin{equation}\label{tau_Q}
		T_Q \define {\inf}\{n\in\mathbb{N}_0 : \pi_{n,0} \leq Q\},
		\end{equation}
		for a properly chosen $Q$ as $\rho \to 0$.
	\end{Theorem} 
	\begin{proof}
	See Appendix \ref{proof:TRP_CONV}.
	\end{proof}
	
	\noindent
	In the following section, we propose a procedure denoted as the radial propagation (RP) procedure for single change-point detection, which is based on the threshold test from \eqref{tau_Q}.
	
	\section{Single change-point detection procedure for radial propagation}\label{sec:Single change-point detection procedure for radial propagation}
	%The posterior probability of change-point or event occurrence is a commonly-used test statistic for performing sequential change-point detection \cite{TARTAKOVSKY_GENERAL,NITZAN_TSP_BAYESIAN,GENG_BAYRAKTAR}. We define the posterior probability of the event $\{t \leq n \}$ using $I_n$ as
	%\be\label{posterior_cp_big}
	%\Pi_{n} \define \Prob(t \leq n | I_n)=1- \Prob(t > n | I_n),~n=0,1,2,\ldots,
	%\ee
	%where $\Pi_{0}=\rho$. The event $\{t > n \}$ implies that $R_n = 0$ for any $O=o_m\in\mathcal{O}$. Thus, using \eqref{Omega_partition}-\eqref{event_partition}, we can rewrite \eqref{posterior_cp_big} as
	%\be\label{posterior_cp_min}
	%\Pi_{n} = 1- \Prob(R_n = 0 | I_n) = 1-\pi_{n,0},~n\in\mathbb{N}_0,
	%\ee
	%where the second equality is obtained by substituting \eqref{posterior_cp}.\\
	%\indent
	In the previous section, it was observed that in the limit $\rho \to 0$, the optimal Bayesian stopping rule converges to a simple threshold rule $T_Q$, defined in \eqref{tau_Q}. In this subsection, we study the performance of this stopping rule for any $\rho$. 
	%T_{\text{RP}}\define{\inf}\left\{n\in\mathbb{N}_0 : \Pi_{n}\geq Q= 1-\alpha\right\}. 
	%\ee
	
	The following proposition provides an upper bound for the false alarm of probability of $T_Q$.
	\begin{proposition}\label{PROP_PFA}
		The false alarm probability of $T_{\text{Q}}$ from \eqref{tau_Q} can be upper bounded with ${\text{PFA}}(T_{Q})\leq Q$.
	\end{proposition}
	\begin{proof}
		By combining \eqref{PFA} and \eqref{finite_horizon_eq}, one obtains
		\begin{equation*}
		\begin{split}
		{\text{PFA}}(T_Q) &= \Prob(R_{T_Q} = 0) = \Ex\left[\Prob(R_{T_{Q}} = 0 |I_{T_{Q}})\right] \\
		&= \Ex[\pi_{T_Q,0} | I_{T_Q}] \leq Q,
		\end{split}
		\end{equation*}
		where the second equality is obtained using the law of iterated expectations, the third equality is obtained from the definition of $\pi_{n,0}$ in \eqref{posterior_cp} and the inequality from the definition $T_Q$.
	\end{proof}
	\indent
	\begin{remark}
		\normalfont It should be noted that the PFA upper bound of Proposition 2 is valid even if many of the model assumptions are violated. As $T$ is a stopping time, $\{T < t\} \in I_{t-1}$. As all observations in $I_{t-1}$ are generated from the pre-change model, it is clear from the definition of the probability of false alarm in (5) that the PFA depends only on the pre-change observations. Violations of the assumed post-change behavior, such as a misspecified $f_1$ or departures from the assumed propagation model do not impact the PFA. This is a useful property, since the post-change distributions (usually generated by signal + noise) are often more difficult to characterize than the pre-change (noise only), as training data may be available from the pre-change probability model only.
	\end{remark}
	
	From here on, we refer to the stopping time $T_Q$ as the radial propagation (RP) procedure, where the stopping threshold is chosen to equal the false alarm upper bound $\alpha$,
	\begin{equation}\label{key}
	T_\text{RP}  \define\inf\{n \in \mathbb{N}_0 : \pi_{n,0} \leq \alpha\}.
	\end{equation}
	
	\subsection{Asymptotic optimality}
	In full generality, the ADD of the RP procedure is tedious to analyse due to the unknown source origin point, the potential mobility of the sensors and their arbitrary locations at each time slot. In order to shed some light on the ADD of the RP procedure, we provide sufficient conditions under which the RP procedure is asymptotically optimal in the vanishing PFA regime $\alpha\to 0$.\\
	\indent
	In the asymptotic analysis we consider the case where the disruption propagates in a deterministic fashion with constant velocity, i.e. $\rho_1 = 1$, so that its area radius increases by one unit in each time slot up to the maximum radius, $\mfR$. The observations $\Xmat_n$ are conditionally independent with pre-change pdf $f_0$ and post-change pdf $f^{(a, o_m)}_1$, respectively, where $f_1^{(a, o_m)}$ may depend on the sensor location $a$ and source location $o_m$. 
%	Additionally, we assume that the number of sensors is fixed and equal to $|\mathcal{A}_n| = L>0,~n\in\mathbb{N}$. 
	To proceed, let us define for all $n$ the set $\eta(n,k, m) \define \{a \in \mathcal{A}_n : \lVert a-o_m \rVert < n-k+1\}$ that 
%	includes locations of sensors at time $n$, whose distance from source $o_m$ is smaller than $n-k+1$. Thus, $\eta(n,k,m)$ 
	contains the locations of sensors that observe the event at time $n$ assuming it took place at time $k$ at origin $m$. Since the event propagates with constant velocity, 
	%the set  $\eta(n,k)$ is well-defined conditional on $k$.
	on $\{t=k, O = o_m\}$, the joint density of the observations received at time $n \geq k$ by the FC is
	\begin{equation}\label{joint_likelihood}
	f_{1,k,m}(\xvec_n)\define\prod_{a \in \eta(n,k,m)} f_1^{(a, o_m)}(x_n^{(a)}) \prod_{a \notin \eta(n,k,m)}f_0(x_n^{(a)}),
	\end{equation}
	with the factorization given is a result of the conditional independence of the sensor data.
	For $\{t=\infty\}$, at time $n$ the joint pdf of the observations received by the FC is $f_{0}(\xvec_n)\define\prod_{a \in \mathcal{A}_n}f_0(x_n^{(a)})$.
The log-likelihood ratio of the hypotheses $\{t = k, O = o_m\}$ and $\{t=\infty\}$ at time $n$ is:
	\be\label{Z_n_k}
	Z_{n}^{k,m} \define \sum_{i=k}^{n}\log\frac{f_{1,k,m}(\xvec_i)}{f_{0}(\xvec_i)},~k\leq n.
	\ee
	In order to analyze the asymptotic detection delay, some conditions on the long-term behaviour of the log-likelihood ratio process $Z_{n}^{k,m}$ are required. It is assumed that there exists some $q_m$ such that for every $k$ and $m$, on $\{t = k, O = o_m\}$
	\be\label{Z_convergence}
	\frac{1}{n}Z_{k+n}^{k,m} \longrightarrow q_m \quad \mbox{ almost surely}.
	\ee
	Note that if the post-change distribution is independent of location, i.e. $f_1^{(a, o_m)} = f_1$, and the number of sensors remains constant over time $|\mathcal{A}_n| = L$,  it follows from the strong law of large numbers, the i.i.d. assumption and finiteness of $\mfR$ that $q_m = L \cdot D(f_1 || f_0)$ for all $m$. The following Lemma provides an asymptotic lower bound for the ADD for any procedure $T$ that fulfills PFA($T$) $\leq \alpha$.
	\begin{lemma}\label{THEOREM_ADD_LOWER}
		Suppose that $\rho_1 = 1$  and that \eqref{Z_convergence} applies for all $m \in [M]$. Let $\Delta_\alpha\define\{T:\emph{PFA}(T)\leq\alpha\}$. Then,
		\be\label{OPT_ADD_LOWER_BOUND}
		\underset{T\in\Delta_\alpha}{\inf}{\emph{ADD}}(T) \geq \frac{1}{M} \sum_{m=0}^{M-1} \frac{|\log \alpha|}{q_m + |\log(1-\rho)|}(1 + o(1)),
		\ee
		where $o(1)\to 0$ as $\alpha \to 0$.
	\end{lemma}
	\begin{proof}
From the definition of ADD we have that \begin{equation}\label{first_decomp}
\text{ADD}(T) = \frac{1}{M}\sum_{m=0}^{M-1} \text{ADD}_m(T),
\end{equation} where $\text{ADD}_m(T) \define \Ex \left[(T-t)^+ | O = o_m \right]$ is the detection delay when the true source location is $o_m$. Conditional on the source location $o_m$ being known, the problem reduces to a standard Bayesian quickest detection formulation with a non-i.i.d. post-change distribution given by \eqref{joint_likelihood}. A lower bound for the asymptotic detection delay of any procedure in the class $\Delta_\alpha$ in this setting was derived in \cite{TARTAKOVSKY_GENERAL}. Specifically, by \cite[Thm. 1]{TARTAKOVSKY_GENERAL}
\begin{equation}\label{single_LB}
\underset{T\in\Delta_\alpha}{\inf}{\text{ADD}_m}(T) \geq \frac{|\log \alpha|}{q_m + |\log(1-\rho)|}(1 + o(1)).
\end{equation}
The Lemma follows from combining \eqref{first_decomp} and \eqref{single_LB}.
	\end{proof}

	It should be noted that the asymptotic lower bound from \eqref{OPT_ADD_LOWER_BOUND} is identical to the lower bound for the case of instantaneous change, where all the sensors are affected at the same time \cite{VEERAVALLI_DECENTRALIZED}.\\
	\indent
	The almost sure converge of $Z_n^{k,m}$, as required in \eqref{Z_convergence}, is not sufficient for proving the asymptotic optimality of the threshold rule $T_{\text{RP}}$. Therefore, in the following theorem we impose some mild additional assumptions on the rate of convergence of ${Z}_n^{k,m}$ to $q_m$ and show that $T_{\text{RP}}$ is asymptotically optimal and attains the lower bound from \eqref{OPT_ADD_LOWER_BOUND}. To this end, we define for $\epsilon>0$ the random variable,
	\begin{equation*}
	Q_\epsilon^{(k,m)}=\sup\left\{n\in\mathbb{N} : \left|\frac{1}{n}Z_{k+n-1}^{k,m}-q_m\right|>\epsilon \right\},
	\end{equation*}
	which is the largest value of $n$ for which the absolute difference between $\frac{1}{n}Z_{k+n-1}^{k,m}$ and $q_m$ is larger than $\epsilon$. It is required that
	\be\label{cond_converge_Z}
	\sum_{k=1}^{\infty}\Prob(t=k){\Ex}[Q_\epsilon^{(k,m)}|t=k]<\infty,~\forall\epsilon>0, m \in [M].
	\ee
	Similarly to \cite[Eq. (3.22)]{TARTAKOVSKY_GENERAL}, the condition in \eqref{cond_converge_Z} is a joint condition on the convergence rates of $\frac{1}{n}Z_{k+n}^{k,m}$ for each $t=k$ and the prior distribution of the change point $t$. In particular, it is analogous to complete convergence \cite{HSU} of $\frac{1}{n}Z_{t+n}^{t,m}$ to $q_m$ under the distribution of $t$.
	\begin{Theorem}\label{T_RP_OPT}
		Suppose the conditions of Theorem \ref{THEOREM_ADD_LOWER} are satisfied and assume that \eqref{cond_converge_Z} is satisfied. Then $T_{\emph{RP}}$ is first-order asymptotically optimal in the limit $\alpha \to 0$, i.e.
		\be\label{asymptotic_ratio}
		\underset{\alpha\to 0}{\lim}\frac{\underset{T\in\Delta_\alpha}{\inf}{\emph{ADD}}(T)}{{\emph{ADD}}(T_{\emph{RP}})}=1,
		\ee
		where
		\be\label{asymptotic_ADD_equality}
		\underset{T\in\Delta_\alpha}{\inf}{\emph{ADD}}(T) = \frac{1}{M} \sum_{m=0}^{M-1} \frac{|\log \alpha|}{q_m + |\log(1-\rho)|}(1 + o(1)).
		\ee
	\end{Theorem}
	\begin{proof}
	The proof is given in Appendix \ref{proof:TRP_OPT}.
	
%		The assumed conditions in the theorem coincide with the conditions in \cite[Corollary 1]{TARTAKOVSKY_GENERAL} for \cite[Eq. (3.47)]{TARTAKOVSKY_GENERAL} to hold. This immediately yields \eqref{asymptotic_ratio}-\eqref{asymptotic_ADD_equality}.
	\end{proof}
	\noindent
	
	\subsection{Detection of attenuating signals}\label{subsec:fading}
	In this subsection, we show how the obtained asymptotic results can be used to accurately approximate the expected detection delay in practically relevant settings. We consider the case of detecting an attenuating random Gaussian signal in additive noise. This model is highly relevant in a variety of practical applications in e.g. wireless communications and radar \cite{LAI_COGNITIVE}.  Prior to the change, only zero-mean i.i.d. Gaussian noise with variance $\sigma^2$ is observed. At an unknown time $t$, a signal source becomes active somewhere in the field. If the signal does not have any known structure, it can be modelled as zero-mean Gaussian with variance $\gamma^2$, where $\gamma^2$ is the signal transmit power. In free space, radio wave power decreases as the inverse square of distance $d$ between the source and the receiver \cite{RADIO_BOOK}. In most practical wireless settings, the path loss exponent, denoted here by $\theta$, is usually greater than 2 due to obstacles, reflectors and scatterers. Therefore, for a sensor at distance $d$ away from the source excluding antenna and frequency dependent factors, the observed signal is of the form $\mathcal{N}(0, \gamma^2/\tilde{d}^\theta)$, where $\tilde{d} = \max(d,1)$. Denoting $f_1^{(d)}$ as the post-change distribution at distance $d$ from the source, we have 
	$f_0 = \mathcal{N}(0, \sigma^2)$ and $f_1^{(d)} = \mathcal{N}(0, \sigma^2 + \gamma^2/\tilde{d}^\theta)$. Suppose for analysis purposes that the signal source location is known, that the domain of interest is a disk with large radius $R$ centered at the signal source, and that at each time step there are $L$ sensors located independently and uniformly at random within the disk. The following result establishes that $T_\text{RP}$ is asymptotically optimal in this setting, and provides a first order approximation of the asymptotic detection delay.
	
\begin{proposition}\label{prop:fading}
Under the conditions described in Section \ref{subsec:fading}, $T_\emph{RP}$ is first-order asymptotically optimal. Moreover, in the free-space conditions of path-loss exponent $\theta = 2$ 
\begin{equation}\label{add_gaussian}
\emph{ADD}(T_\emph{RP}) = \frac{|\log \alpha|}{Lq_\phi + |\log(1-\rho)|}(1 + o(1)),
\end{equation}
where $q_\phi = \frac{1}{2R^2}\left[\phi + \phi\log(\phi +1) - (\phi + R^2)\log\left(1+ \frac{\phi}{R^2}\right)\right]$ and $\phi = \gamma^2/\sigma^2$ is the SNR in linear scale.
\end{proposition}
\begin{proof}
The proof is provided in Appendix \ref{proof:prop_fading}.
\end{proof}
	Observe, that the expression for $L q_\phi$ in \eqref{add_gaussian} further simplifies in the limit $R^2, L \to \infty$. When $R^2, L \to \infty$ such that $L/R^2 \to \lambda$, we have
	\begin{equation}\label{final_form}
	Lq_\phi \overset{R, L \to \infty}{\longrightarrow }\frac{\lambda}{2}\phi\log(\phi +1),
	\end{equation}
where $\lambda$ represents the average number of sensors per unit area.
	
	\section{Extension to multiple change-point detection}\label{sec:Extension to multiple change-point detection}
	In this subsection, we briefly describe how the simple structure of the RP stopping rule allows its use in settings when one is monitoring multiple separate fields and signal sources at once. This is a very relevant case in practice since in IoT, wireless networks or radar systems there may be multiple active signal sources, abrupt events or targets simultaneously and there is a need to strictly control the false positives in decision making while detecting changes rapidly. Suppose that there are $K\geq 2$ distinct clusters of sensors. For each cluster, $k\in[K]$, there may exist a random initial event (change point), at time $t^{(k)}$, that propagates and affects the sensors in the cluster according to the model in Section \ref{sec:Model and problem formulation}. We allow the probability of no event in a cluster to be non-zero, where no event implies an infinite change point. The spatial events and sensor observations of the different clusters are assumed to be independent. This assumption, while restrictive in general, is reasonable in cases where the sensor clusters exist in spatially dispersed locations, and the events are spatially localized. The assumed setup is illustrated in Fig. \ref{wavefront_cluster}. We would like to derive multiple stopping rules, $T^{(k)},~k\in[K]$, in order to discover all the change points, $t^{(k)},~k\in[K]$, respectively, while strictly controlling Type I errors. We employ a sequential multiple hypothesis testing framework for this purpose. \\
	\indent
	A practical assumption for any sequential detection procedure is that it must be stopped at some finite time instance. Thus, we allow the existence of a deadline $N_{\text{max}}$ for the multiple change-point detection. If a change point in the $k$th cluster has not been declared before time slot $N_{\text{max}}$, we declare that there is no spatial event in the $k$th cluster and set $T^{(k)}=\infty$. For detecting multiple change-points in parallel, the False Discovery Rate (FDR) is a relevant false alarm rate criterion \cite{CHEN_ZHANG_POOR_BAYESIAN_JOURNAL,NITZAN_TSP_BAYESIAN}. This criterion is defined as
	\be\label{FDR_define}
	{\text{FDR}}\define {\Ex}\bigg[\frac{ V}{\max (R, 1)}\bigg].
	\ee
	The term $V$ is the number of false discoveries (false alarms) under deadline, i.e. the size of the subset of $[K]$ s.t. $T^{(k)}<t^{(k)}$ and $T^{(k)}<N_{\text{max}}$. The term $R$ denotes the number of discoveries under deadline, i.e. the size of the subset of $[K]$ s.t. $T^{(k)}<N_{\text{max}}$. We would like to control the FDR s.t. it will be no higher than a predefined tolerated level $\alpha\in(0,1)$.\\
	\indent
	Taking into account the possibility of infinite change points, we denote by $K_f$ the random number of finite change points and define the overall ADD as
	\be\label{ADD}
	{\text{ADD}}\define{\Ex}\left[\frac{1}{K_f}\sum_{k,t^{(k)}<\infty}(T^{(k)}-t^{(k)})^+\right],
	\ee
	where for $K_f=0$ the argument of the expectation in \eqref{ADD} is zero. In case the {\em no change-point} probabilities are zero we can rewrite the ADD as
	\be\label{ADD_K}
	{\text{ADD}}\define \frac{1}{K}\sum_{k=1}^{K}{\Ex}[(T^{(k)}-t^{(k)})^+].
	\ee
	\indent
	For the considered multiple statistically independent clusters, we will implement the following $K$ parallel stopping rules:
	\be\label{CARP_K}
	T_{\text{RP}}^{(k)}\define{\inf}\left\{n\in\mathbb{N}_0 : \pi_{n,0}^{(k)} \leq \alpha\right\},~k\in[K],
	\ee
	where $\pi_{n,0}^{(k)}$ is the posterior probability of a cluster change point having occurred in cluster $k$. A cluster change point is the time of an initial event in the cluster. The threshold choice in \eqref{CARP_K} guarantees FDR control under upper bound $\alpha$, in accordance with the parallel version of the IS-MAP procedure in \cite{NITZAN_TSP_BAYESIAN}.\\
	\indent
	In each cluster it is assumed that there is no change point with probability $p_\infty$ and with probability $1-p_\infty$ the prior distribution of the initial change point, $t^{(k)}$, is geometrically distributed with parameter $\rho$. Under the above assumptions, the change-point posterior probability update is similar to the one described in Subsection \ref{subsec:posterior} and implemented for each cluster separately. However, some expressions for probabilities need to be rederived. For simplicity of presentation, we omit the cluster index, $k\in[K]$, in the following expressions. For a specific cluster, by using \eqref{geom_prior} and the no change-point probability, $p_\infty$, we obtain
	\be\label{R_0_1_to_R_0_multiple}
	\begin{split}
		\Prob(R_{n} = 1 | R_{n-1} = 0) &= P(t=n|t\geq n) \\
		&=\rho\frac{(1-p_\infty)(1-\rho)^{n-1}}{p_\infty + (1-p_\infty)(1-\rho)^{n-1}},
	\end{split}
	\ee
	$n\in\mathbb{N}$, where we recall that $P(R_{n} = 1 | R_{n-1} = 0) = 1 - P(R_{n} = 0 | R_{n-1} = 0)$. In addition to \eqref{R_0_1_to_R_0_multiple}, the following expressions are rewritten to take into account the no change-point probability:
	\begin{equation*}
	\Prob(R_{0} = 0)=p_\infty+(1-p_\infty)(1-\rho),  
	\end{equation*}
	\begin{equation*}
	\Prob(R_{0} = 1)=(1-p_\infty)\rho,
	\end{equation*}
	\begin{equation*}
	\Prob(U_{n,m,1}|U_{n-1,m,0})=\rho\frac{(1-p_\infty)(1-\rho)^{n-1}}{p_\infty + (1-p_\infty)(1-\rho)^{n-1}}, 
	\end{equation*}
	\begin{equation*}
	\Prob(U_{n,m,0}|U_{n-1,m,0})=\frac{p_\infty + (1-p_\infty)(1-\rho)^{n}}{p_\infty + (1-p_\infty)(1-\rho)^{n-1}},
	\end{equation*}
	\begin{equation*}
	\Prob(U_{0,m,1})=\frac{1}{M}(1-p_\infty)\rho, 
	\end{equation*}
	and 
	\begin{equation*}
	\Prob(U_{0,m,0})=\frac{1}{M}(p_\infty+(1-p_\infty)(1-\rho)),
	\end{equation*}
	for $m \in [M]$.
	
	\begin{figure}[ht!]
		\centering
		\includegraphics[scale = 1]{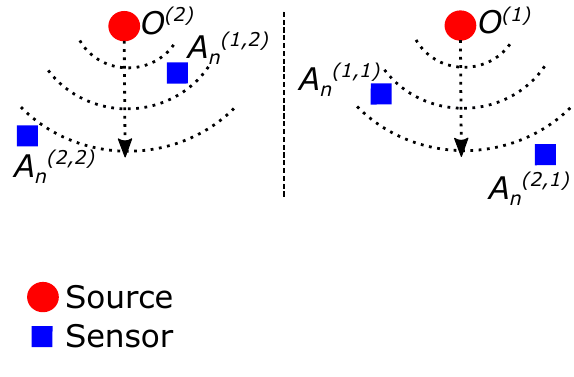}
		\caption{Sources wavefront propagation: Phenomena emanating in two distinct clusters of sensors from the sources $O^{(1)}$ and $O^{(2)}$ towards the sensors at locations $A_n^{(1,1)},A_n^{(2,1)}$ and $A_n^{(1,2)},A_n^{(2,2)}$, respectively.}
		\label{wavefront_cluster}
	\end{figure}

	\section{Numerical simulations}\label{sec:Numerical simulations}
	In this section, we evaluate the performance of the RP procedure in terms of PFA and FDR control and ADD performance under different radial propagation models and different multi-sensor configurations. 
%	Since the extension of the RP procedure to parallel detection of events in multiple clusters is relatively straightforward, in the simulations we only consider event detection in a single-cluster for simplicity. The general takeaways obtained from the single-cluster setting apply also to the case of multiple independent clusters monitored in parallel. 
In order to better understand the behavior of the RP procedure, we compare its performance to other procedures that either know the unobservable true source location, or deploy a more simplistic propagation model. Robustness to misspecification is tested by implementing a misspecified RP procedure that incorrectly assumes the event to propagate much faster than it does.
	
	%\begin{figure}[ht!]
	%\centering´
	%\includegraphics[scale = 0.5]{figs/scenario.eps}
	%\caption{Illustration of the assumed scenario: Cluster area is $\mathcal{S}=[0,1]\times [0,1]$, source origin point is $(0,0)$, and the cluster is composed of two sensors located at $(0.1,0.1)$ and $(0.9,0.9)$. The radius unit is $0.5$, i.e. when $R=1,2,3$ the event area radii are $0.5,1,1.5$, respectively, and $\mfR=3$ covers the entire cluster area.}
	%\label{scenario}
	%\end{figure}
	
	\subsection{Simple Gaussian observation model}\label{subsec:Simple Gaussian observation model}
	We begin by considering a single cluster and a simple Gaussian observation model where $f_0 = \mathcal{N}(0,1)$ and $f_1 = \mathcal{N}(0,1+\gamma^2)$, no matter the sensor location and the event origin point. In all experiments $L = 100$ sensors are randomly placed on the field at each time instance. The true source location is selected randomly from a uniform distribution over the field. Observe, that this is in contrast to the design-stage assumption that the true source locations lies in the finite set $\mathcal{O}$. The RP procedure is compared against two other procedures. The first one is an Oracle version of the RP procedure that knows the exact source location. The Oracle procedure is a special case of the RP procedure with $|\mathcal{O}| = 1$. The other procedure implemented for comparison purposes assumes that the event, once it appears, affects all sensors instantly \cite{VEERAVALLI_DECENTRALIZED}.  We refer to this procedure as the Instant procedure. 
%	The stopping time is again given by a threshold rule of the form $T= \inf\{n: \Pi_n \leq \alpha\}$, where $\alpha$ is the PFA upper bound, and $\Pi_n$ is the posterior probability of no event being present,
%	\be\label{recursive_update}
%	\begin{split}
%		&\Pi_{n}=\\
%		&\frac{\prod_{l=1}^L f_0(x_n^{(l)})(1-\kappa(\Pi_{n-1}))}{\prod_{l=1}^L f_1(x_n^{(l)})\kappa(\Pi_{n-1}) + \prod_{l=1}^L f_0(x_n^{(l)})(1-\kappa(\Pi_{n-1}))},
%	\end{split}
%	\ee
%	where $\kappa(y)\define y+\rho(1-y)$ and $\Pi_0=1-\rho$. 
	The Instant procedure is also a particular special case of the RP procedure, where one assumes $R_n = \mathfrak{R}$ for $n\geq t$ and $R_n = 0$ for $n < t$. It is to be expected that this procedure will provide inferior performance to the RP procedure, as it does not take the dynamic nature of the propagation into account. However, the comparison will provide insight into to the behaviour of the RP procedure by highlighting the scenarios in which the performance gap between the properly specified and misspecified procedures is significant, and where the difference in performance is smaller.
	
	We start by setting $\rho = 0.02$, $\rho_1 = 0.25$, $\gamma^2 = 1$ and considering a square spatial field $\mathcal{S} = [0,10] \times [0,10]$ where the sensors and sources are located. The set $\mathcal{O}$ used by the RP procedure is taken to be an equally spaced grid of $M$ points which covers the field of interest. In addition to the properly specified RP procedure, we implement a mismatched RP stopping rule (with $M = 50$), which correctly assumes that radius increases with probability $\rho_1$ but with increments of 5 times the true radius increment (1 unit). It corresponds to a setting where the real event propagates slower than assumed by the RP procedure.
	
	\begin{table}
		\centering
		\caption{Observed false alarm probabilities for different threshold values $\alpha$. 
%		The RP procedures for all $M$ control the false alarm probability below $\alpha$, hence supporting Proposition \ref{PROP_PFA}.
}
		\label{table1}
		\begin{tabular}{lrrrr}
			\toprule
			\multicolumn{1}{l}{}&\multicolumn{1}{c}{$\alpha = 0.1$}&\multicolumn{1}{c}{$\alpha = 0.05$}&\multicolumn{1}{c}{$\alpha = 0.01$}&\multicolumn{1}{c}{$\alpha = 0.005$}\tabularnewline
			\midrule
			RP, $M = 10$&$0.040$&$0.020$&$0.002$&$0.000$\tabularnewline
			RP, $M = 50$&$0.034$&$0.012$&$0.002$&$0.002$\tabularnewline
			RP, $M = 100$&$0.036$&$0.022$&$0.004$&$0.002$\tabularnewline
			RP, mismatched&$0.032$&$0.016$&$0.001$&$0.001$\tabularnewline
			Oracle&$0.024$&$0.016$&$0.004$&$0.004$\tabularnewline
			Instant&$0.008$&$0.002$&$0.001$&$0.001$\tabularnewline
			\bottomrule
		\end{tabular}
	\end{table}

	In Table \ref{table1}, observed false alarm probabilities of all procedures for different stopping thresholds $\alpha$ are displayed. It is confirmed that the theoretical PFA upper bound derived in Proposition \ref{PROP_PFA} holds in all cases. In the top plot of Figure \ref{fig:ex1}, the PFA-ADD trade-off curves are plotted for the procedures, with the RP procedure implemented using source location grids of density $M =$ 10, 50, and 100. For this small field, the performance of the RP procedure is comparable to the Oracle procedure. Furthermore, it is observed that increasing the density of the location grid in the RP procedure improves performance. However, under this configuration for $M = 50$ and $M = 100$ the gap in performance is already indistinguishable. All versions of RP procedure, including the misspecified one, outperform the Instant procedure.  In the middle plot of Figure \ref{fig:ex1}, the procedures are compared for varying values of the propagation parameter $\rho_1$, with $\alpha = 0.01$ fixed. For small values of $\rho_1$, the Instant procedure experiences performance loss in comparison to the others. This is because when $\rho_1$ is small the event will expand slowly with respect to the discrete-time sampling rate and thus remain spatially localized for a longer time, making it harder to detect for the Instant procedure. In general there is an inverse relationship between ADD and $\rho_1$ for all procedures, as a larger $\rho_1$ implies that the event will be visible to more sensors quicker. The RP and Oracle procedures provide near identical performance for all $\rho_1$. The mismatched RP procedure achieves lower ADD than the Instant procedure for all $\rho_1$ values.  In the bottom plot of Fig \ref{fig:ex1}, we fix $\rho_1 = 0.25$, $\alpha = 0.1$ and vary the signal power parameter $\gamma^2$. For unit noise variance, we have SNR (dB) = $10\log_{10}(\gamma^2)$. It is observed that at low SNR regime the difference between the RP and Instant procedures is smaller, but for moderate and and high SNRs a clear gap in performance in favor of the RP procedure again emerges. Moreover, the difference in performance between the RP and Oracle procedures is small, and the size of the perfomance gap is relatively independent of SNR.
	
		\begin{figure}[ht!]
		\centering
		\input{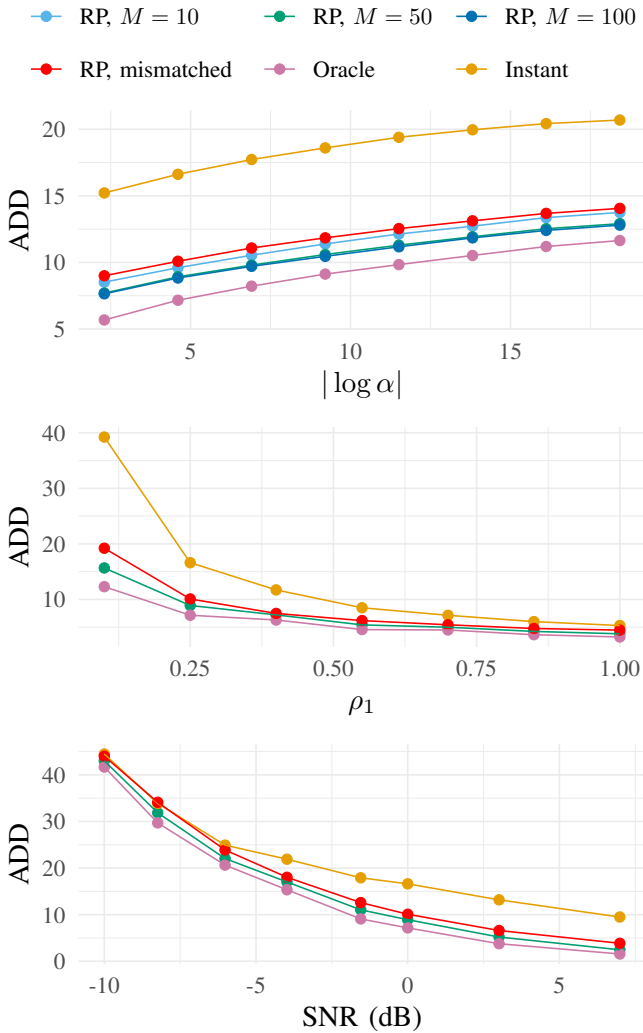}
		\caption{Top: Average detection delay as a function of the PFA bound $\alpha$ for all procedures. 
%		It is observed, that by properly taking the propagation dynamics into account, the RP procedure provides superior performance to the Instant procedure. 
%		Furthermore, the gap in performance between the RP procedure, and an Oracle procedure that is aware of the true source location is small, and independent of $\alpha$. 
			Middle: ADD evaluated for different values of propagation parameter $\rho_1$, while keeping other parameters fixed. 
%			All procedures detect the event quicker as $\rho_1$ increases. A large $\rho_1$ implies that the event expands quicker, and therefore is observable by more sensors faster. The performance gap between the Instant procedure and the RP procedure is the largest when $\rho_1$ is small. 
%			The RP procedure provides near identical performance to the Oracle procedure for all values of $\rho_1$.
			Bottom: ADD as a function of the SNR.
			 %Naturally, the ADD decreases as SNR increases. 
%			 For very small SNR the RP procedure and Instant procedure provide similar performance, but as SNR increases, a clear gap in performance appears.
		}
		\label{fig:ex1}
	\end{figure}

	\subsection{Detection of attenuating radio signals}\label{subsec:Fading signals}
	In this subsection, we implement the attenuating signal model introduced in Subsection \ref{subsec:fading}. To demonstrate the extension of the RP procedure to the detection of multiple events in parallel, we consider a setting with $K = 20$ distinct, independent sensor clusters each with 100 sensors as described in Sec. \ref{sec:Extension to multiple change-point detection}. Prior to the change in a given cluster, all sensors observe noise only, so that $f_0 = \mathcal{N}(0,1)$. When a signal source appears in the $k$th cluster at time $t^{(k)}$, it starts emitting an i.i.d. random signal modeled as $\mathcal{N}(0,\gamma^2)$. Due to path loss, the received signal strength attenuates according to a path-loss exponent $\theta$ of the distance $d$ from the source. The signal and noise are considered additive, hence for a sensor at distance $d$ from the signal source we have $f_1^{(d)} = \mathcal{N}(0, 1 + \gamma^2/\tilde{d}^\theta)$, where $\tilde{d} = \max(d_0,d/d_0)$ with $d_0$ being a reference distance where the received signal power equals $\gamma^2$. In the case of radio waves, the signal propagates at the speed of light $c$. The sensors take discrete time samples with some common sampling rate $f_s$. Therefore, the signal area radius expands in a deterministic manner (i.e. $\rho_1 = 1$) by $c/f_s$ meters in a single time step. We take each cluster area $\mathcal{S}_k, k \in [K]$ to be a square field with side length 5 km. The time at which the signal appears, $t^{(k)}$, is considered to have an exponential prior distribution with a mean (in seconds) of $\beta = 10$ in all sensor clusters. A routine computation utilizing the properties of the exponential and geometric distributions then shows that the sample index at which the emitted signal first appears obeys a geometric distribution with parameter $\rho = 1-\exp(-1/(\beta f_s))$. The RP procedure is again compared against an Oracle procedure that knows the true signal source location in each cluster, and the exact propagation dynamics. Additionally, two versions of the Instant procedure are implemented. The first one (called Instant-Oracle) knows the true and unobservable source location in each cluster, but assumes that the event reaches all sensors in the cluster immediately. The other one (Instant) assumes similarly to the RP procedure that the source location in each cluster belongs to a finite set $\mathcal{O}$, and that the change is immediate everywhere in the cluster. Note that in the setting of Subsection \ref{subsec:Simple Gaussian observation model} knowledge of the true source location is not utilized in the Instant procedure since the appearance of the event was assumed to immediately change all sampling distributions from $f_0$ to $f_1$ no matter the source location. However, in this setting, the post-change sampling distribution $f_1^{(d)}$ depends on the distance of the sensor from the source. Therefore, knowing the true source location has value even if the propagation is assumed immediate. Consequently, we obtain an interesting comparison between the RP and the Instant-Oracle procedures, as the RP procedure is aware of the propagation dynamics, but the Instant-Oracle has knowledge of the true source location.
	
	In all clusters, we set the signal power $\gamma^2 = 2$ at a reference distance of 500m from the source, $\alpha = 0.01$ and the path loss exponent $\theta = 2$. The true source location of each cluster is sampled uniformly at random from $\mathcal{S}_k$. The sensor locations are also random and uniform, and assumed to remain stationary during the monitoring process. In Figure \ref{fig:ex2}, the procedures are compared for different values of the sampling rate $f_s$.  The detection delay decreases for all procedures as the sampling rate increases, and the difference in ADD (in microseconds) between the RP and Oracle procedures shrinks as the sampling rate increases. It is observed, that for sufficiently high sampling rates the RP procedure achieves smaller detection delay than the Instant-Oracle procedure. When the sampling rate is high, accounting for the propagation dynamics is more valuable than theoretical knowledge of the source location, and vice versa when the propagation is rapid in comparison to the sampling rate. In Table \ref{table2}, the observed FDR values are displayed for different choices of the stopping threshold $\alpha$ when $f_s = 1$ MHz. It is demonstrated that the RP procedure controls the FDR below the prespecified level $\alpha$.
	
			\begin{table}
		\centering
		\caption{Observed false discovery rates for different threshold values $\alpha$. The RP procedure controls the FDR level below the specified threshold.}
		\label{table2}
		\begin{tabular}{lrrrr}
			\toprule
			\multicolumn{1}{l}{}&\multicolumn{1}{c}{$\alpha = 0.1$}&\multicolumn{1}{c}{$\alpha = 0.05$}&\multicolumn{1}{c}{$\alpha = 0.01$}&\multicolumn{1}{c}{$\alpha = 0.005$}\tabularnewline
			\midrule
			RP, $M = 10$&$0.062$&$0.036$&$0.002$&$0.001$\tabularnewline
			Oracle&$0.048$&$0.020$&$0.006$&$0.005$\tabularnewline
			Instant-Oracle&$0.042$&$0.024$&$0.004$&$0.002$\tabularnewline
			Instant&$0.062$&$0.026$&$0.001$&$0.001$
\tabularnewline
			\bottomrule
		\end{tabular}
	\end{table}

%	In the lower plot of Figure \ref{fig:ex2}, we return to the single cluster setting and evaluate the ADD-PFA tradeoff slopes for the procedures. We set $\rho_1 = 0.7$ and $\gamma^2 = 1$, with other parameter values as earlier. In Subsection \ref{subsec:fading} it was shown that if $\rho_1 = 1$ and the source location is known, the RP procedure is asymptotically optimal. Moreover, an analytic approximation for the slope of the ADD against $|\log\alpha|$ was derived in Proposition  \eqref{prop:fading}. The coefficient $\lambda \define L/R^2$ that appears in eq. \eqref{final_form} is the ratio of the average number of sensors $L$ that are located inside a disk of radius $R$, and $R^2$. Since in this setting there are $100$ sensors randomly located on a square of side length $10$, inside any disk of radius $R$ located inside the square there are on average $L = 100\cdot(\pi R^2/10^2)$ sensors. Therefore we obtain $\lambda = \pi$. From the lower plot of Figure \ref{fig:ex2} it is seen that analytic approximation of the asymptotic ADD slope accurately matches the observed slope of the RP procedure already for small but non-vanishing $\alpha$. Furthermore, the analytic approximation is accurate even though the source location is not known, and $\rho_1 \not=1$. 
	
	\begin{figure}[ht!]
		\centering
		\input{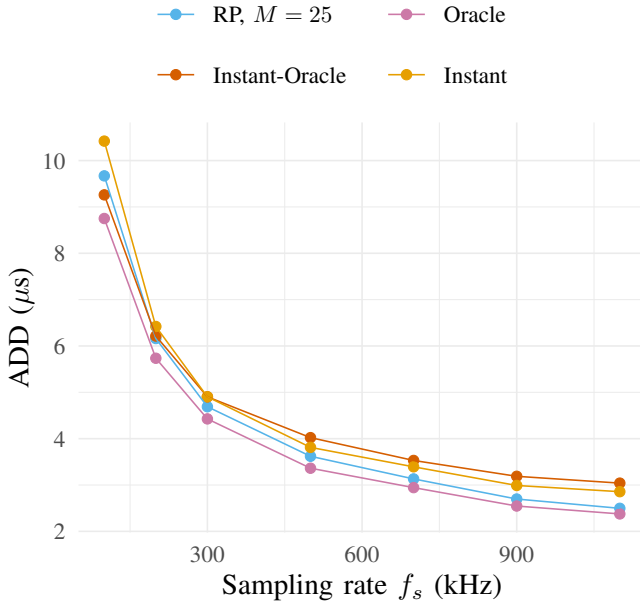}
		\caption{
			Top: ADD evaluated for different values of propagation parameter $f_s$ in the attenuating signal and multiple cluster setting.
%			All procedures detect the event quicker as $\rho_1$ increases. For small values of $\rho_1$ the RP procedure outperforms even the Instant-Oracle procedures, despite the Instant-Oracle utilizing unobservable information. The difference in ADD between the RP and the Oracle procedures remains approximately constant for all $\rho_1$, as was also observed in Fig. \ref{fig:ex1}.
%			Bottom: ADD-PFA tradeoff curves for all procedures for $\gamma^2  =1$ and $\rho_1 = 0.7$. 
%			While the analytic approximation of the ADD-PFA slope from \eqref{integral_result} was derived under the assumptions that $\rho_1 = 1$ and that the true source location is known, it is observed, that the approximation accurately describes the ADD-PFA relationship of the RP procedure. 
		}\label{fig:ex2}
	\end{figure}

	\section{Conclusion}\label{sec:Conclusion}
	In this paper, we proposed a method for Bayesian quickest detection of spatial events with radial propagation patterns using a mobile sensor network. First, we considered a single spatial event. A dynamic programming framework was used to derive the structure of the optimal stopping time in terms of ADD under upper bound constraint on the PFA. The optimal procedure has a complicated structure and implementing an approximation is computationally challenging and infeasible to analyze. Therefore, utilizing a limiting form of the optimal procedure we proposed the simpler RP procedure that employs a stopping threshold on the posterior probability of the change point of interest. It was shown both analytically and experimentally that the RP procedure controls the PFA under a prespecified upper bound, even if the post-change probability models are misspecified. In addition, we showed that under some conditions the proposed RP procedure coincides with an asymptotically optimal procedure in terms of ADD as the PFA upper bound $\alpha\to 0$. Then, we proposed an extension to parallel detection of multiple spatial events occurring in distinct clusters. The proposed method stems from a multiple hypothesis testing problem formulation and strictly controls FDR criterion while taking into account the spatial nature of the observed phenomena or fields. A posterior probability update expression for multiple change-point detection which takes into account a probability that no event appears was derived.\\
	\indent
	In the simulations it was observed that for phenomena that propagate slowly with regard to the sampling rate, the RP procedure vastly outperforms a procedure that assumes that the effect takes place instantly everywhere in the field. Similarly, in the high SNR regime the RP procedure provided significantly better performance than the Instant procedure. When the event propagates very quickly in relative to the sampling rate, or alternatively the SNR is very low, the performance gap was smaller, although still in favor of the RP procedure.
	
	Topics for future research include the derivation of spatial procedures for multiple change-point detection and localization, where the locations of the signal sources are estimated using the observations. Additionally, extending the RP procedure to a non-Bayesian framework and studying its possible optimality properties is an interesting direction of future work.
	
\appendices

\section{Proof of Theorem \ref{TRP_CONV}}\label{proof:TRP_CONV}

Stemming from \cite[Th. 2]{RAGHAVAN}, our proof proceeds by showing that the optimal stopping time $T_\text{opt}$ can be written as 
		\begin{equation}\label{threshold_structure}
		T_\text{opt} = \inf \left\{n\in\mathbb{N}_0 : \pi_{n,0} < \frac{c + \Psi_n}{c+\rho}\right\},
		\end{equation}
		where $\Psi_n$ is a function such that $\frac{\Psi_n}{\rho} \to 0$  as $\rho \to 0$. The desired threshold test structure from \eqref{tau_Q} is then obtained in the limit $\rho \to 0$. Let us define 
		\begin{equation}\label{psieq}
		\Psi_n \define \mathcal{D}_n(\bm{p}_n) - (1-\rho)\pi_{n,0}.
		\end{equation}
		Substituting this definition into \eqref{optimal_stop} and rearranging gives \eqref{threshold_structure}. Then, the convergence of $\frac{\Psi_n}{\rho}\to 0$, as $\rho \to 0$, can be shown by introducing the transformation
		\begin{equation*}
		q_{n,r} = \frac{\pi_{n,r}}{\rho\pi_{n,0}} \iff \pi_{n,r} = \frac{q_{n,r}}{\sum_{\tilde{r}=0}^{\mfR} q_{n,\tilde{r}}}.
		\end{equation*}
		This expression allows for  using the steps in \cite[Th. 2]{RAGHAVAN} to complete the proof.

\section{Proof of Theorem \ref{T_RP_OPT}}\label{proof:TRP_OPT}
	
The proof is in two parts. First we define a set of $M$ stopping times $T^{(0)},...,T^{(M-1)}$, such that 
\begin{equation}
T^{(m)} \define \inf\{n : \Wnm \geq \nu  \},
\end{equation}
for some threshold $\nu$, where $\Wnm \define   \Prob (t \leq n | I_n, O = o_m)$. In Lemma \ref{lemma:multiple_UB} below, we show that a stopping time defined as the minimum of these $M$ stopping times with thresholds $\nu = 1-\alpha/M$ achieves the asymptotic ADD lower bound. Then, it is shown that $\text{ADD}(T_\text{RP}) \leq \text{ADD}(T^*)$, and the Theorem follows. 

\begin{lemma}\label{lemma:multiple_UB}
Let $T^* = \inf\{T^{(0)},...,T^{(M-1)} \}$ and $\nu = 1 - \alpha /M$. Then \begin{equation}\label{LEMMA1}
\emph{ADD}(T^*) \leq \frac{1}{M}\sum_{m = 0}^{M-1}\frac{|\log \alpha|}{q_m + |\log(1-\rho)|}(1+ o(1)),
\end{equation}
i.e. $T^{*}$ achieves the asymptotic ADD lower bound in \eqref{OPT_ADD_LOWER_BOUND}.
\end{lemma}
\begin{proof}
Observe first from the definition of ADD that \begin{equation}\label{decomp}
\text{ADD}(T^*) = \frac{1}{M}\sum_{m=0}^{M-1} \text{ADD}_m(T^*),
\end{equation} where $\text{ADD}_m(T^*) \define \Ex \left[(T^*-t)^+ | O = o_m \right]$. When $o_m$ is the true source location, the problem reduces to a Bayesian quickest detection task with a non-i.i.d. post-change distribution. It is shown in \cite{TARTAKOVSKY_GENERAL}, that $T^{(m)}$ with a properly chosen stopping threshold is asymptotically optimal for minimizing ADD$_m$, so that
\begin{equation}\label{individual_bound}
\text{ADD}_m\left(T^{(m)}\right) \leq \frac{|\log (1-\nu)|}{q_m + |\log(1-\rho)|}(1+ o(1)).
\end{equation}
Since $|\log (1-\nu)| = |\log \alpha| + \log M = |\log \alpha|(1 + o(1))$, and by definition $T^* \leq T^{(m)}$ for all $m$, combining \eqref{decomp} and \eqref{individual_bound} yields \eqref{LEMMA1}.
\end{proof}
By Proposition \ref{PROP_PFA}, PFA$(T_\text{RP}) \leq \alpha$. Therefore, in order to prove Theorem \ref{T_RP_OPT}, it is sufficient to show that $\text{ADD}(T_\text{RP}) \leq \text{ADD}(T^*)$. We have
\begin{equation}
1 - \pi_{n,0} = \sum_{m=0}^{M-1} \Wnm \Prob(O = o_m | I_n),
\end{equation}
and 
\begin{equation}\label{PO_decomp}
\begin{split}
&\Prob(O = o_m | I_n) = \Prob(O = o_m, t \leq n |I_n) + \Prob(O = o_m, t > n |I_n) \\
&=\Wnm \Prob(O = o_m | I_n) + \frac{\pi_{n,0}}{M}.\\
\end{split}
\end{equation}
The second equality in \eqref{PO_decomp} follows from the fact that given $t > n$, the event $\{O = o_m\}$ is independent of $I_n$. Rearranging, one obtains $\Prob(O = o_m | I_n) = \pi_{n,0}/(M(1-\Wnm))$. Hence,
\begin{align}
1 - \pi_{n,0} &= \frac{\pi_{n,0}}{M}\sum_{m=0}^{M-1} \frac{\Wnm}{1-\Wnm}, \\
\pi_{n,0} &= \frac{M}{M + \sum_{m=0}^{M-1}\frac{\Wnm}{1-\Wnm}}.
\end{align}
Since the function $h(x) = M / (M+x)$ is decreasing in $x$ for $x > 0$, and $W_{T^*}^{(m)} \geq \nu = 1 - \alpha/M$ for some $m$ by definition of $T^*$, we obtain
\begin{equation}\label{pi_decomp}
\begin{split}
&\pi_{T^*,0} \leq \frac{M}{M + \frac{1-\alpha/M}{\alpha/M}} = \frac{M\alpha}{M\alpha + M - \alpha} \\
&= \alpha + \frac{\alpha^2 -  M\alpha^2  }{M\alpha + M - \alpha},
\end{split} 
\end{equation}
where the first equality is obtained by rearrangement, and the second equality by adding and subtracting $\alpha$ and rearranging. Since $M \geq 1$ and $\alpha \in [0,1]$, the second term on the last line of \eqref{pi_decomp} is non-positive, and hence $\pi_{T^*, 0} \leq \alpha$. As $T_\text{RP} = \inf\{n : \pi_{n,0} \leq \alpha\}$ and $\pi_{T^*, 0} \leq \alpha$, we obtain $T_\text{RP} \leq T^{*}$ and $\text{ADD}(T_\text{RP}) \leq \text{ADD}(T^*)$.

\section{Proof of Proposition \ref{prop:fading}}\label{proof:prop_fading}

We would like to use Theorem 4 to establish asymptotic optimality of $T_{\text{RP}}$ for quickest detection of the propagating signal, and approximate the detection delay in this setting. 
	
	To compute the constant $q$ defined in \eqref{Z_convergence} and appearing in \eqref{asymptotic_ADD_equality}, observe that we have the partitioning
	
	\be\label{Z_n_k_part}
	Z_{k+n}^k = \sum_{i=k}^{k+\mfR-1}\log\frac{f_{1,k}(\xvec_i)}{f_{0}(\xvec_i)}+\sum_{i=k+\mfR}^{k+n}\log\frac{f_{1,k}(\xvec_i)}{f_{0}(\xvec_i)},
	\ee
	where $R$ is the number of time steps needed for the event to cover the entire region. On $\{t = k\}$ when $i \geq k + R$, the signal reaches all sensors, no matter where they are located within the domain. Therefore, for $i \geq k + R$
	\begin{equation}\label{log_lik_form}
	\mathbb{E} \left[\log\frac{f_{1,k}(\xvec_i)}{f_{0}(\xvec_i)}\right] = L\mathbb{E}_{f_1^{(d)}}\left[\log \frac{f_1^{(d)}(x)}{f_0(x)}\right],
	\end{equation}
	where the latter expectation is over both the random location (in particular the random distance $d$ from the source) and the random observation generated from the post-change $f_1^{(d)}$ distribution.  Since $R$ is a finite constant, by the strong law of large numbers and \eqref{Z_n_k_part}-\eqref{log_lik_form}
	\begin{equation}\label{Z_ktoEf}
	\frac{1}{n} Z_{k+n}^{k} \underset{n \to \infty}{\longrightarrow} L\mathbb{E}_{f_1^{(d)}}\left[\log \frac{f_1^{(d)}(x)}{f_0(x)}\right] \quad {\text{a.s.}}
	\end{equation}
	By iterated expectation and a direct computation of the KL-divergence between two Gaussians we get
	\begin{equation}
	\mathbb{E}_{f_1^{(d)}}\left[\log \frac{f_1^{(d)}(x)}{f_0(x)}\right] = \frac{1}{2}\mathbb{E}_d\left[\frac{\phi}{\tilde{d}^\theta} - \log\left(1+ \frac{\phi}{\tilde{d}^\theta}\right)\right] \define q_\phi,
	\end{equation}
	where $\phi \define \gamma^2 / \sigma^2$ is the SNR in linear scale. As the sensor locations are uniform on the disk, $\Prob(d \leq s) =  (s/R)^2, \mbox{ for } 0 \leq s \leq R$.  Therefore,
	\begin{equation}\label{integral}
	q_\phi = \frac{1}{2R^2}\int_0^{R} \frac{\phi}{\tilde s^{\theta-1}} - s\log\left(1+ \frac{\phi}{\tilde s^\theta}\right) ds.
	\end{equation}
	Evaluating the integral in \eqref{integral} for a general path loss exponent $\theta$ is possible, but leads to convoluted results. In the commonly considered free-space conditions of $\theta = 2$, we have 
	\begin{equation}\label{integral_result}
	q_\phi = \frac{1}{2\mfR^2}\left[\phi + \phi\log(\phi +1) - (\phi + R^2)\log\left(1+ \frac{\phi}{R^2}\right)\right]. 
	\end{equation}
	
To apply Theorem 4, it remains to check that the joint convergence condition of both the observations and the prior distribution in \eqref{cond_converge_Z} is satisfied. This is straightforward, since on $\{t = k\}$, $Z_{k+n}^{k+\mfR}$ is a sum of i.i.d. random variables such that $n^{-1}Z_{k+n}^k$ converges almost surely to \eqref{final_form}, and $\mfR$ is a finite constant. Therefore, following \cite[Sec. 4]{TARTAKOVSKY_GENERAL}, \eqref{cond_converge_Z} is established. The Proposition then follows from Theorem 4 in this Appendix.
	
	\bibliographystyle{IEEEtran}
	\bibliography{refs}
\end{document}